\documentclass[12pt]{amsart}
\usepackage{amssymb,amsmath,amsthm,mathrsfs,multirow,enumerate}

\usepackage{graphicx} 
\usepackage[all]{xy}
\usepackage{cite}

	\usepackage[bottom=1in, top=1in, left=1in, right=1in]{geometry}

\numberwithin{equation}{section} 

\theoremstyle{plain}
\newtheorem{thm}{Theorem}[section]

\newtheorem{lemma}[thm]{Lemma}

\theoremstyle{definition}

\newcommand{\bbm}{\begin{bmatrix}}
\newcommand{\ebm}{\end{bmatrix}}
\DeclareMathOperator{\Aut}{\rm Aut}

\DeclareMathOperator{\Sym}{\rm Sym}
\DeclareMathOperator{\Sp}{\rm Sp}

\begin{document}
\title[Orthogonal graphs modulo power of $2$] {Orthogonal graphs modulo power of $2$} 
\author{Songpon Sriwongsa}

\address{Songpon Sriwongsa\\ Department of Mathematics \\ Faculty of Science\\ King Mongkut's University of Technology Thonburi(KMUTT) \\Bangkok 10140, Thailand}
\email{\tt songpon.sri@kmutt.ac.th, songponsriwongsa@gmail.com}



\begin{abstract}
In this work, we define an orthogonal graph on the set of equivalence classes of $(2\nu + \delta)-$tuples over $\mathbb{Z}_{2^n}$ where $n$ and $\nu$ are positive integers and $\delta = 0, 1$ or $2$. We classify our graph if it is strongly regular or quasi-strongly regular and compute all parameters precisely. We show that our graph is arc transitive. The automorphisms group is given and the chromatic number of the graph except when $\delta = 0$ and $\nu$ is odd is determined. Moreover, we work on subconstituents of this orthogonal graph.
\end{abstract}

\date{\today}

\maketitle
\section{Introduction}

Graphs defined from the geometry of classical groups over finite fields have been widely studied. The collinearity
graphs of finite classical polar spaces and their complements are well known strongly
regular graphs \cite{BL84, H75}. For more details about strongly regular graphs, the reader is referred to \cite{GR01}. Orthogonal graphs over finite fields of odd characteristic using the geometry of orthogonal group and automorphisms of the graphs were studied in \cite{Gu08, Liebeck85}. 
Then, in \cite{GWZ11, GWZ13}, subconstituents of these orthogonal graphs were analyzed. Li et al. \cite{LW14} studied the orthogonal graphs over Galois rings of odd characteristic using matrix theory over finite Galois rings. Recently, Meemark and Sriwongsa extended this work to the orthogonal graphs finite commutative rings of odd characteristic \cite{Me16}. The analogous series of symplectic graphs and their subconstituents were presented in \cite{Gu13,GW13,LW08,LWG12,LWG13,Liebeck85, MP11,MP13,MP14,Tang06}. Motived by these references, in this paper, we consider the analogous problems of the orthogonal graphs modulo power of $2$ which is a generalization of the orthogonal graph over a finite field $\mathbb{Z}_2$ \cite{Zhe09}. 

The paper is organized as follows. We define and study orthogonal graphs modulo power of $2$ in Section \ref{s2}. This includes the computation on number of vertices , degree, common neighbors and chromatic numbers. We also analyze arc transitivity and automorphism groups of the graphs. In Section \ref{s3}, subconstituents of the orthogonal graphs modulo power of $2$ are studied. We use  similar technique as in Section \ref{s2} to determine all parameters of these subconstituents. 

\section{Orthogonal graphs modulo $2^n$}\label{s2}
Let $n$ be a positive integer. For $\nu \geq 1$ and $\delta = 0, 1$ or $2$, let $V^{2\nu + \delta}$ denote the set of $(2\nu + \delta)$- tuples $\vec{a} = (a_1, a_2, \ldots, a_{2\nu + \delta})$ of elements in $\mathbb{Z}_{2^n}$ such that $a_i$ is invertible modulo $2^n$ for some $i \in \{1, 2, \ldots, 2\nu + \delta \}$. Define an equivalence relation $\sim$ on $V^{2\nu + \delta}$ by
\[
(a_1, a_2, \ldots, a_{2\nu + \delta}) \sim (b_1, b_2, \ldots, b_{2\nu + \delta}) \Leftrightarrow (a_1, a_2, \ldots, a_{2\nu + \delta}) = \lambda (b_1, b_2, \ldots, b_{2\nu + \delta})
\]
for some $\lambda \in \mathbb{Z}_{2^n}^\times$. Here $\mathbb{Z}_{2^n}^\times$ is the unit group modulo $2^n$. Write $[\vec{a}] = [a_1, a_2, \ldots, a_{2\nu + \delta}]$ for the equivalence class of $\vec{a} = (a_1, a_2, \ldots, a_{2\nu + \delta})$ modulo $\sim$, and let $V^{2\nu + \delta}_\sim$ be the set of all such equivalent classes. Let $G_{2 \nu + \delta, \Delta}$ be the $(2 \nu + \delta) \times (2 \nu + \delta)$ matrix over $\mathbb{Z}_{2^n}$ given by
\[
G_{2\nu + \delta, \Delta} = \begin{pmatrix}
0 & I_\nu &  \\
   & 0 & \\
&        & \Delta
\end{pmatrix},
\]
where
\[
\Delta =
\begin{cases}
\emptyset (\text{disappear}) & \text{if} \ \delta = 0, \\
(1) \     & \text{if} \ \delta = 1, \\
\begin{pmatrix}
z & 1 \\
 0      & z
\end{pmatrix}        & \text{if} \ \delta = 2,
\end{cases} 
\]
and $z$ is a fixed element of $\mathbb{Z}_{2^n}$ not in $N = \{x^2 + x : x \in \mathbb{Z}_{2^n} \}$, or equivalently, $z \in \mathbb{Z}_{2^n}^\times$ because $N = 2\mathbb{Z}_{2^n}$ and $\mathbb{Z}_{2^n} = \mathbb{Z}_{2^n}^\times \cup 2\mathbb{Z}_{2^n}$.

The {\it orthogonal graph modulo $2^n$} on $V^{2\nu + \delta}_\sim$, denoted by $\mathcal{O}^{(2\nu + \delta)}_{2^n}$, is the graph whose vertex set is  $ \mathcal{V}(\mathcal{O}^{(2\nu + \delta)}_{2^n}) = \{ [\vec{a}] \in V^{2\nu + \delta}_\sim : \vec{a} G_{2\nu + \delta, \Delta} \vec{a} ^t = 0 \}$ and its adjacency condition is given by
\[
 [\vec{a}] \text{ \ is adjacent to \ } [\vec{b}] \iff  \vec{a} (G_{2\nu + \delta, \Delta} + G_{2\nu + \delta, \Delta}^t) \vec{b}^t \in \mathbb{Z}_{2^n}^\times.
  \]
 To see that this adjacency condition is well defined, let $\vec{a}_1, \vec{a}_2, \vec{b}_1,\vec{b}_2 \in V^{2\nu + \delta}$ and assume that $[\vec{a}_1] = [\vec{a}_2]$ and $[\vec{b}_1] = [\vec{b}_2]$. Then $\vec{a}_1 = \lambda \vec{a}_2$ and $\vec{b}_1 = \lambda' \vec{b}_2$ for some $\lambda, \lambda' \in\mathbb{Z}_{2^n}^\times $. Thus, we have
 \begin{align*}
 \vec{a}_1 (G_{2\nu + \delta, \Delta} + G_{2\nu + \delta, \Delta}^t) \vec{b}_1^t \in \mathbb{Z}_{2^n}^\times 
 & \Leftrightarrow \lambda\lambda' \vec{a}_2 (G_{2\nu + \delta, \Delta} + G_{2\nu + \delta, \Delta}^t) \vec{b}_2^t \in \mathbb{Z}_{2^n}^\times  \\ & \Leftrightarrow  \vec{a}_2 (G_{2\nu + \delta, \Delta} + G_{2\nu + \delta, \Delta}^t) \vec{b}_2^t \in \mathbb{Z}_{2^n}^\times.
 \end{align*}
It is obvious that this orthogonal graph is a generalization of the orthogonal graph over $\mathbb{Z}_2$ defined similarly in \cite{Zhe09}.
 
We recall that $\mathbb{Z}_{2^n}$ is a finite local ring with the unique maximal ideal $ 2\mathbb{Z}_{2^n}$ and the residue $\mathbb{Z}_2 = \mathbb{Z}_{2^n}/2\mathbb{Z}_{2^n}$. Therefore, $u + m$ is a unit in $\mathbb{Z}_{2^n}$ for all $m \in 2\mathbb{Z}_{2^n}$ and $u \in \mathbb{Z}_{2^n}^\times$. Moreover, $|2 \mathbb{Z}_{2^n}| = 2^{n - 1}$. The following lemma is an important property of a vertex of $\mathcal{O}^{(2\nu + \delta)}_{2^n}$.

\begin{lemma}\label{vertex}
	If $[a_1, a_2, \ldots, a_{2\nu + \delta}]$  is a vertex in $\mathcal{O}^{(2\nu + \delta)}_{2^n} $, then $a_i \in \mathbb{Z}_{2^n}^\times$ for some $i \in \{ 1, 2, \ldots, 2\nu \}$.
\end{lemma}
\begin{proof}
 The result is directed when $\delta=0$. If $\delta = 1$ and $a_1, a_2, \dots, a_{2\nu} \in 2\mathbb{Z}_{2^n}$, then 
 \[
  0 = a_1 a_{\nu + 1} + a_2 a_{\nu + 2}+ \cdots + a_\nu a_{2\nu} + a^2_{2\nu + 1}
 \]
 implies $a_{2\nu+1}$ is also an element in $2\mathbb{Z}_{2^n}$ which is a contradiction. Now assume that $\delta = 2$ and $a_1, a_2, \dots, a_{2\nu} \in 2\mathbb{Z}_{2^n}$. Then 
 \[
 0= a_1 a_{\nu + 1} + a_2 a_{\nu + 2} + \cdots + a_\nu a_{2\nu} + z a^2_{2\nu + 1} + a_{2\nu +1}a_{2\nu + 2} + z a^2_{2\nu + 2}
 \]
 implies $ z a^2_{2\nu + 1} + a_{2\nu +1}a_{2\nu + 2} + z a^2_{2\nu + 2} \in 2\mathbb{Z}_{2^n} $ and so $\bar{z} \bar{a}^2_{2\nu + 1} + \bar{a}_{2\nu +1}\bar{a}_{2\nu + 2} + \bar{z} \bar{a}^2_{2\nu + 2} = 0$ in the residue field $\mathbb{Z}_2$. This forces that $\bar{a}_{2\nu+1} = \bar{a}_{2\nu+2} = 0$, which is impossible. 
\end{proof}

Since $\mathbb{Z}_2$ is the residue field of $\mathbb{Z}_{2^n}$, the matrix $G_{2\nu + \delta, \Delta}$ over $\mathbb{Z}_{2^n}$ induces the matrix $\overline{G}_{2\nu + \delta, \Delta}$ over $\mathbb{Z}_2$ in an obvious manner via the canonical map $\pi : \mathbb{Z}_{2^n} \rightarrow \mathbb{Z}_2$. It follows that 
\[
\vec{a} (G_{2\nu + \delta, \Delta} + G_{2\nu + \delta, \Delta}^t) \vec{b}^t \in \mathbb{Z}_{2^n}^\times \iff  \pi(\vec{a}) (\overline{G}_{2\nu + \delta, \Delta} + \overline{G}_{2\nu + \delta, \Delta}^t) \pi(\vec{b})^t = 1
\]
for all $\vec{a}, \vec{b} \in V^{2\nu + \delta}$. Here, we write 
\[
\pi(\vec{a}) = (\pi(\vec{a}_1), \pi(\vec{a}_2),  \ldots, \pi(\vec{a}_{2\nu + \delta}))
\]
for all $\vec{a} = (a_1, a_2, \ldots, \vec{a}_{2\nu + \delta}) \in  V^{2\nu + \delta} $. In the following lemma, we show that the results on counting all parameters of the orthogonal graph over $\mathbb{Z}_{2^n}$ can be considered as lifts from the graphs over its residue field $\mathbb{Z}_2$.

\begin{lemma}\label{Lifting Lemma} (Lifting Lemma)
	By the above setting, we have the following.
\begin{enumerate}[(1)]
	\item If $[\vec{a}]$ is a vertex of $\mathcal{O}^{(2\nu + \delta)}_{2^n}$, then there are $2^{(n - 1)(2\nu + \delta - 2)}$ many vertices which are lifts of vertex $[\pi (\vec{a})]$ of $\mathcal{O}^{(2\nu + \delta)}_{2}$, i.e. $|\{ [\vec{b}] \in  \mathcal{V}(\mathcal{O}^{(2\nu + \delta)}_{2^n}) : [\pi(\vec{a})] =[\pi(\vec{b})] \}| = 2^{(n - 1)(2\nu + \delta - 2)}$.
	\item 
	$[\vec{a}]$ and $[\vec{b}]$ are adjacent vertices in $\mathcal{O}^{(2\nu + \delta)}_{2^n}$ if and only if $[ \pi(\vec{a})]$ and $[\pi(\vec{b})]$ are adjacent vertices in  $\mathcal{O}^{(2\nu + \delta)}_2$.
	
	\item
	If  $[\pi(\vec{a})]$ and $[\pi(\vec{b})]$ are adjacent vertices in  $\mathcal{O}^{(2\nu + \delta)}_{2}$ , then $[\vec{a} + \vec{m}_1]$ and $[\vec{b} + \vec{m}_2]$ are adjacent vertices in the graph $\mathcal{O}^{(2\nu + \delta)}_{2^n}$ for all $\vec{m}_1, \vec{m}_2 \in (2\mathbb{Z}_{2^n})^{2\nu + \delta}$ such that 
	\[
		(\vec{a} + \vec{m_1})  G_{2\nu + \delta, \Delta} (\vec{a} + \vec{m_1})^t = (\vec{b} + \vec{m_2})  G_{2\nu + \delta, \Delta} (\vec{b} + \vec{m_2})^t = 0.
	\]
\end{enumerate}		
\end{lemma}
\begin{proof}
It is clear that (2) follows from the above discussion and (3) is a consequence of (2). Now, we note that for each vertex  $[\vec{b}]$ of $\mathcal{O}^{(2\nu + \delta)}_{2^n}$ such that $[\pi(\vec{a})] =[\pi(\vec{b})]$, $\vec{b}$ must equal $\vec{a} + \vec{m}$ for some $\vec{m} \in  (2\mathbb{Z}_{2^n})^{2\nu + \delta}$. So we assume that $[\vec{a} + \vec{m}_1] = [\vec{a} + \vec{m}_2]$. Then $\vec{a} + \vec{m}_1 = \lambda (\vec{a} + \vec{m}_2)$ for some $\lambda \in \mathbb{Z}_{2^n}^\times$. Thus, $(1-\lambda) \vec{a} = \lambda \vec{m}_2 - \vec{m}_1 \in ( 2\mathbb{Z}_{2^n})^{2\nu + \delta}$. This implies $1 - \lambda \in 2\mathbb{Z}_{2^n}$ by Lemma \ref{vertex}, so $\lambda = 1 + \mu$ for some $\mu \in 2\mathbb{Z}_{2^n}$. Hence, $\vec{a} + \vec{m}_1 = (1 + \mu)(\vec{a} + \vec{m}_2)$. Next, we remark that $[(1 + \mu)(\vec{a} + \vec{m})] = [\vec{a} + \vec{m}]$ for all $\mu \in 2\mathbb{Z}_{2^n}$, $\vec{a} \in (\mathbb{Z}_{2^n})^{2\nu + \delta}$ and $\vec{m} \in (2\mathbb{Z}_{2^n})^{2\nu + \delta}$. Therefore, the number of elements in the set $\{ [\vec{a} + \vec{m}] : \vec{m} \in (2\mathbb{Z}_{2^n})^{2\nu + \delta}\}$ is $2^{(n - 1) (2\nu+\delta-1)}$. However, the $\vec{a}+ \vec{m}$ is also required that $(\vec{a} + \vec{m})G_{2\nu + \delta, \Delta} (\vec{a} + \vec{m})^t = 0$. Write $\vec{a} = (a_1, a_2, \ldots, a_{2\nu + \delta})$. By Lemma \ref{vertex}, $a_j$ is a unit for some $j \in \{1,2, \dots, 2\nu\}$. The requirement $(\vec{a} + \vec{m})G_{2\nu + \delta, \Delta} (\vec{a} + \vec{m})^t = 0 $ and $a_j$ is a unit allow us to count the number of possible vectors $\vec{m}$ and it is easy to see that $|\{[\vec{a} + \vec{m}] : \vec{m} \in 2\mathbb{Z}_{2^n}^{2\nu + \delta} \ \text{and} \ (\vec{a} + \vec{m})G_{2\nu + \delta, \Delta} (\vec{a} + \vec{m})^t = 0 \}| = 2^{(n - 1)(2\nu+\delta-2)}$. This completes the proof of the lemma.
\end{proof}

A \textit{strongly regular graph} with parameters $(v, k, \lambda, \mu)$ is a $k$-regular graph on $v$ vertices such that for every pair of adjacent vertices there are $\lambda$ vertices adjacent to both, and for every pair of non-adjacent vertices there are $\mu$ vertices adjacent to both. An $k$-regular graph $G$ on $v$ vertices such that for every pair of adjacent vertices there are $\lambda$ vertices adjacent to both is called a {\it quasi-strongly regular} with parameters $(v, k, \lambda, \{c_1,\ldots,c_d\})$ if every pair of non-adjacent vertices of $G$  has $c_1, c_2, \ldots , c_d$ common adjacent vertices for some $d \ge 2$.

 From the work of Wan and Zhou (Theorem 2.4 of \cite{Zhe09}), it follows that the orthogonal graph $\mathcal{O}^{(2\nu + \delta)}_2$  is $2^{2\nu + \delta - 2}$-regular on $(2^\nu - 1)(2^{\nu + \delta - 1} + 1)$ many vertices. Moreover, if $\nu = 1$ , then it is a complete graph, and if $\nu \geq 2$, then the graph is a strongly regular graph with parameters 
\[ 
\lambda = 2^{2\nu + \delta -2} - 2^{2\nu + \delta -3} - 2^{\nu - 1} + 2^{\nu + \delta -2} \text{ and } 
\mu = 2^{2\nu + \delta -2} - 2^{2\nu + \delta -3},
\]  
respectively.

In what follows, we classify our orthogonal graph and it turns out that the graph is quasi-strongly regular when $\nu \geq 2$. In the proof of the next theorem, we use the lifting lemma with some combinatorial arguments without solving any complicated equations.

\begin{thm}\label{vertices}
	The graph $\mathcal{O}^{(2\nu + \delta)}_{2^n}$  is $2^{n(2\nu + \delta - 2)}$-regular on
	\[
	2^{(n - 1)(2\nu + \delta - 2)}(2^\nu - 1)(2^{\nu + \delta - 1} + 1)
	\]
	many vertices. Moreover,
	\begin{enumerate}
		\item If $\nu = 1$, then it is a strongly regular graph with parameters 
		\[ 
		\lambda = 2^{\delta n} - 2^{\delta(n - 1)} \text{ and } 
		\mu =  \lceil \delta/2 \rceil 2^{\delta n},
		\]  
		respectively.
		\item If $\nu \geq 2$, then it is a quasi-strongly regular graph with parameters 
		\begin{align*}
		\lambda &= 2^{n-1}(2^{n(\nu + \frac{\delta}{2}-1)} + (\delta - 1)2^{(n-1)(\nu + \frac{\delta}{2}-1)})2^{n (\nu - 2 + \frac{\delta}{2})}, \\
		c_1 &= 2^{n-1}2^{n(2\nu - 3 + \delta)} \text{ and } 
		c_2 = 2^{n(2\nu - 2 + \delta)}.
		\end{align*}
		respectively.
	\end{enumerate}
\end{thm}
\begin{proof}
We first note that every vertex of $\mathcal{O}^{(2\nu + \delta)}_2$ can be viewed as a vertex of $\mathcal{O}^{(2\nu + \delta)}_{2^n}$ via the canonical map $\pi$. By Lemma \ref{Lifting Lemma} (1), each vertex of $\mathcal{O}^{(2\nu + \delta)}_2$ can be lifted to $2^{(n- 1)(2\nu + \delta - 2)}$ vertices in  $ \mathcal{V}(\mathcal{O}^{(2\nu + \delta)}_{2^n})$. Thus the number of vertices of $\mathcal{O}^{(2\nu + \delta)}_{2^n}$ is 
\[
2^{(n - 1)(2\nu + \delta - 2)}(2^\nu - 1)(2^{\nu + \delta - 1} + 1).
\]
Since the graph $\mathcal{O}^{(2\nu + \delta)}_2$ is $2^{2\nu + \delta - 2}$-regular, Lemma \ref{Lifting Lemma} implies that the graph $\mathcal{O}^{(2\nu + \delta)}_{2^n}$ is also regular of degree
$
2^{2\nu + \delta - 2}2^{(n-1)(2\nu + \delta - 2)} = 2^{n(2\nu + \delta - 2)}.
$ 

Next, we consider the case $\nu = 1$. Since the graph $\mathcal{O}^{(2 + \delta)}_2$ is complete, for each pair of adjacent vertices in the graph $\mathcal{O}^{(2 + \delta)}_{2^n}$, there are 
\[
(2^\delta - 1)2^{(n - 1)\delta} = 2^{\delta n } - 2^{\delta (n-1)}
\]
common neighbors by Lemma \ref{Lifting Lemma} (3).
If $\delta = 0$, then $\mathcal{O}^{2}_2$ is a path on two vertices and so is $\mathcal{O}^{2}_{2^n}$. Hence $\mu = 0$. Now, suppose that $\delta \neq 0$. Any two of non-adjacent vertices are of the forms $[\vec{a} + \vec{m}_1]$ and $[\vec{a} + \vec{m}_2]$ for some $[\vec{a}] \in V^{2\nu + \delta}_\sim$ and $\vec{m}_i \in (2\mathbb{Z}_{2^n})^{2\nu + \delta}$ such that  $(\vec{a} + \vec{m}_i)G_{2\nu + \delta, \Delta} (\vec{a} + \vec{m}_i)^t = 0, i=1, 2$. Thus, for every pair of non-adjacent vertices, the number of common neighbors equals the degree of regularity of $\mathcal{O}^{(2 + \delta)}_{2^n}$ by Lemma \ref{Lifting Lemma} (3), so we have $\mu = 2^{\delta n}$.

Finally, we assume that $\nu \geq 2$. For each pair of adjacent vertices $[\vec{a} + \vec{m}_1]$ and $[\vec{b} + \vec{m}_2]$ in the graph $\mathcal{O}^{(2 + \delta)}_{2^n}$, the number of common neighbors is given by the product of the common neighbors of vertices $[\pi(\vec{a})]$ and $[\pi(\vec{b})]$ and $2^{(n-1)(2\nu + \delta - 2)}$ by Lemma \ref{Lifting Lemma}. Thus,
\begin{align*}
\lambda &= (2^{2\nu + \delta -2} - 2^{2\nu + \delta -3} - 2^{\nu - 1} + 2^{\nu + \delta -2})2^{(n-1)(2\nu + \delta -2)}\\
&= 2^{n - 1}(2^{n(\nu + \frac{\delta}{2}-1)} + (\delta - 1)2^{(n - 1)(\nu + \frac{\delta}{2}-1)})2^{n(\nu - 2 + \frac{\delta}{2})}.
\end{align*}
Assume that $[\vec{a} + \vec{m}_1]$ and $[\vec{b} + \vec{m}_2]$ are non-adjacent vertices in $\mathcal{O}^{(2 + \delta)}_{2^n}$.
If  $\pi (\vec{a}) \neq \pi (\vec{b})$, then $[\pi(\vec{a})]$ and $[\pi(\vec{b})]$ are non-adjacent vertices in $\mathcal{O}^{(2 + \delta)}_2$, so the number of common neighbors of $[\vec{a} + \vec{m}_1]$ and $[\vec{b} + \vec{m}_2]$ is the product of common neighbors of $[\pi(\vec{a})]$ and $[\pi(\vec{b})]$ and $2^{(n-1)(2\nu + \delta -2)}$ which equals $2^{n - 1}2^{n (2\nu + \delta - 3)}$ by Lemma \ref{Lifting Lemma}. If $\pi (\vec{a}) = \pi (\vec{b})$, the number of common neighbors is the degree of regularity of $\mathcal{O}^{(2 + \delta)}_{2^n}$ by a similar reason to the last sentence of the previous paragraph.
\end{proof}

A graph $G$ is \textit{vertex transitive} if its automorphism group acts transitively on the vertex set. That is, for any two vertices of $G$, there is an automorphism carrying one to the other. An \textit{arc} in $G$ is an ordered pair of adjacent vertices, and $G$ is \textit{arc transitive} if its automorphism group acts transitively on its arcs.

\begin{lemma}\label{OVtran} (See \cite{Zhe09}.)
	 The orthogonal graph $\mathcal{O}^{(2\nu + \delta)}_2$ is vertex transitive and arc transitive.
\end{lemma}

For any set $A$, we denote the set of all permutations of $A$ by $\Sym (|A|)$. For any vertex $[\vec{a}]$ of $\mathcal{O}^{(2\nu + \delta)}_{2^n}$, let
\[
X_{\vec{a}} := \{ [\vec{a} + \vec{m}] : \vec{m} \in (2\mathbb{Z}_{2^n})^{2\nu + \delta} \text{ \ and \ } (\vec{a} + \vec{m})  G_{2\nu + \delta, \Delta} (\vec{a} + \vec{m})^t = 0  \}.
\]
This is the set of all vertex in $\mathcal{O}^{(2\nu + \delta)}_{2^n}$ which are the lifts of the vertex $[\pi(\vec{a})]$ of 
$\mathcal{O}^{(2\nu + \delta)}_2$ and
 $\prod\limits_{[\pi(\vec{a})] \in \mathcal{V}(\mathcal{O}^{(2\nu + \delta)}_2)} X_{\vec{a}} = \mathcal{V}(\mathcal{O}^{(2\nu + \delta)}_{2^n})$.
 It is easy to see that each permutation in $\Sym(|X_{\vec{a}}|)$ can be regard as an automorphism of our graph. Then we have:

\begin{thm}\label{tran}
	The orthogonal graph $\mathcal{O}^{(2\nu + \delta)}_{2^n}$ is vertex transitive and arc transitive.
	\begin{proof}
		It suffices to show that $\mathcal{O}^{(2\nu + \delta)}_{2^n}$ is arc transitive. 
		Let $[\vec{a} + \vec{m}_1], [\vec{b} + \vec{m}_2], [\vec{c} + \vec{m}_3], [\vec{d} + \vec{m}_4]$ be vertices of $\mathcal{O}^{(2\nu + \delta)}_{2^n}$ such that 
		\begin{align*}
		&[\vec{a} + \vec{m}_1] \text{ is adjacent to } [\vec{c} + \vec{m}_3] \text{ and } \\
		&[\vec{b} + \vec{m}_2] \text{ is adjacent to } [\vec{d} + \vec{m}_4]. 
		\end{align*}
		 By applying four suitable permutations, we may assume that $\vec{m}_1 = \vec{m}_2 = \vec{m}_3 = \vec{m}_4 = \vec{0}$. Note that vertices $[\vec{a} ], [\vec{b}], [\vec{c}], [\vec{d}]$ in $\mathcal{V}(\mathcal{O}^{(2\nu + \delta)}_{2^n})$ correspond with $[\pi(\vec{a}) ], [\pi(\vec{b})], [\pi(\vec{c})], [\pi(\vec{d})]$ in $\mathcal{V}(\mathcal{O}^{(2\nu + \delta)}_2)$, respectively. By Lemma \ref{OVtran}, the result follows directly.
	\end{proof}
\end{thm}

Next, we determine the automorphism group of the orthogonal graph $\mathcal{O}^{(2\nu + \delta)}_{2^n}$ which is another application of the lifting lemma.

\begin{thm}\label{localAut}
The automorphism group of the graph $\mathcal{O}^{(2\nu + \delta)}_{2^n}$ is given by
	\[
	\Aut (\mathcal{O}^{(2\nu + \delta)}_{2^n}) \cong \Aut (\mathcal{O}^{(2\nu + \delta)}_2) \times (\Sym(2^{(n-1)(2\nu + \delta - 2)}))^{(2^\nu - 1 )(2^{\nu + \delta - 1} + 1)},
	\]
	where $\Aut (\mathcal{O}^{(2\nu + \delta)}_2)$ is determined in \cite{Zhe09}.
\end{thm}
\begin{proof}
	By Lemma \ref{Lifting Lemma}, the graph $\mathcal{O}^{(2\nu + \delta)}_{2}$ is isomorphic to a subgraph of $\mathcal{O}^{(2\nu + \delta)}_{2^n}$. Then each automorphism of $\mathcal{O}^{(2\nu + \delta)}_{2^n}$ corresponds with an automorphism of the graph $\mathcal{O}^{(2\nu + \delta)}_2$ and a permutation of vertices in the set $X_{\vec{a}}$ for all vertices $[\pi(\vec{a})] \in \mathcal{V}(\mathcal{O}^{(2\nu + \delta)}_{2})$ . Thus, 
	\begin{align*}
	\Aut (\mathcal{O}^{(2\nu + \delta)}_{2^n}) &\cong \Aut (\mathcal{O}^{(2\nu + \delta)}_2) \times  \prod_{[\pi(\vec{a})] \in \mathcal{V}(\mathcal{O}^{(2\nu + \delta)}_{2})} \Sym (|X_{\vec{a}}|) \\
	&= \Aut (\mathcal{O}^{(2 + \delta)}_2) \times (\Sym(2^{(n-1)(2\nu + \delta - 2)}))^{(2^\nu - 1 )(2^{\nu + \delta - 1} + 1)}
	\end{align*}
	because $|X_{\vec{a}}| = 2^{(n-1)(2\nu + \delta - 2)}$ and $|\mathcal{V}(\mathcal{O}^{(2\nu + \delta)}_{2})| = (2^\nu - 1 )(2^{\nu + \delta - 1} + 1)$. 
\end{proof}

Finally, we study chromatic number of orthogonal graph $\mathcal{O}^{(2\nu + \delta)}_{2^n}$, where $\nu \geq 1$ and $\delta = 0, 1$ or $2$ except when $\delta = 0$ and $\nu$ is odd. The {\it chromatic number} is the smallest number of colors needed to color the vertices of the graph so that no two adjacent vertices share the same color. For convenience, we let $S = \{(\nu, \delta) \in \mathbb{N} \times \{0, 1, 2\} : (\nu, \delta) \neq (2k - 1, 0) \text{ \ for all \ } k \in \mathbb{N} \}$.

 Let $(\nu, \delta) \in S$. According to Wan and Zhou \cite{Zhe09}, the result of the graph $\mathcal{O}^{(2\nu + \delta)}_{2}$ is completely known and $\mathcal{O}^{(2\nu + 1)}_{2}$ is isomorphic to the symplectic graph $\Sp^{(2\nu + 1)}_{2}$. From the proof of Proposition 2.7 along with Proposition 2.9 in \cite{Zhe09}, and Proposition 2.3 in \cite{Tang06}, it follows that $\mathcal{O}^{(2\nu + \delta)}_{2}$ is a $(2^{\nu + \delta - 1} + 1) -$partite graph. Moreover, the chromatic number of $\mathcal{O}^{(2\nu + \delta)}_{2}$ is $2^{\nu + \delta - 1} + 1$. Now, by Lemma \ref{Lifting Lemma}, we know that $\mathcal{O}^{(2\nu + \delta)}_{2^n}$ is a $(2^{\nu + \delta - 1} + 1) -$partite graph as well. Then its chromatic number is at most $2^{\nu + \delta - 1} + 1$. On the other hand, we consider the subgraph of $\mathcal{O}^{(2\nu + \delta)}_{2^n}$ which is isomorphic to $\mathcal{O}^{(2\nu + \delta)}_{2}$. Clearly, its chromatic number is $2^{\nu + \delta - 1} + 1$. Therefore, the chromatic number of $\mathcal{O}^{(2\nu + \delta)}_{2^n}$ is exactly $2^{\nu + \delta - 1} + 1$ and we record in the following theorem.

\begin{thm}\label{chromatic}
For $(\nu, \delta) \in S$, the chromatic number of the orthogonal graph $\mathcal{O}^{(2\nu + \delta)}_{2^n}$ is $2^{\nu + \delta - 1} + 1$.
\end{thm}

\section{Subconstituents of graphs}\label{s3}
In this section, two subconstituents of our orthogonal graphs are studied. This topic for such subconstituents over a finite field of characteristic two, in particular over $\mathbb{Z}_2$, is presented in \cite{GWZ16}. We use the analog method of \cite{Me16} together with our lifting lemma (Lemma \ref{Lifting Lemma}) to obtain the results for these subconstituents.

Let $\vec{e}_1$ denote the vector $(1, 0, \ldots, 0)$. From Theorem \ref{vertices}, we know that our orthogonal graphs are either strongly regular or quasi-strongly regular. Thus, the distance $d([\vec{a}], [\vec{e}_1]) = 1$ or $2$ if $[\vec{a}] \neq [\vec{e}_1]$. We work on the {\it subconstituents} $\mathcal{O}^{(2\nu + \delta)}_{2^n}(i), i = 1, 2$, which are defined to be the induced subgraphs of $\mathcal{O}^{(2\nu + \delta)}_{2^n}$ on the vertex sets
\[
\mathcal{V}_i := \{ [\vec{a}] \in \mathcal{V}(\mathcal{O}^{(2\nu + \delta)}_{2^n}) : d([\vec{a}], [\vec{e}_1]) = i \}
\]
$i = 1, 2$, respectively. Due to the definition, $\mathcal{V}_1$ consists of all adjacent vertices of $[\vec{e}_1]$ while $\mathcal{V}_2$ consists of all non-adjacent vertices of $[\vec{e}_1]$. We observe that it is possible to define another subconstituents associated with other vertices. However, our graph $\mathcal{O}^{(2\nu + \delta)}_{2^n}$ is vertex and arc transitive by Theorem \ref{tran}. Therefore, it suffices to consider only the ones associated with $[\vec{e}_1]$.

\begin{lemma}\label{lifting2}
	For $i = 1, 2$, we have the following statements.
	\begin{enumerate}[(1)]
		\item If $[\vec{a}]$ is a vertex of $\mathcal{O}^{(2\nu + \delta)}_{2^n}(i)$, then there are $2^{(n - 1)(2\nu + \delta - 2)}$ many vertices in $\mathcal{V}_i$ which are lifts of vertex $[\pi(\vec{a})]$ of $\mathcal{O}^{(2\nu + \delta)}_{2}(i)$.
		\item 	$[\vec{a}]$ and $[\vec{b}]$ are adjacent vertices in $\mathcal{O}^{(2\nu + \delta)}_{2^n}(i)$ if and only if $[ \pi(\vec{a})]$ and $[\pi(\vec{b})]$ are adjacent vertices in  $\mathcal{O}^{(2\nu + \delta)}_2(i)$.
		\item
		If  $[\pi(\vec{a})]$ and $[\pi(\vec{b})]$ are adjacent vertices in  $\mathcal{O}^{(2\nu + \delta)}_{2}(i)$ , then $[\vec{a} + \vec{m}_1]$ and $[\vec{b} + \vec{m}_2]$ are adjacent vertices in the graph $\mathcal{O}^{(2\nu + \delta)}_{2^n}(i)$ for all $\vec{m}_1, \vec{m}_2 \in (2\mathbb{Z}_{2^n})^{2\nu + \delta}$ such that
		\[
			(\vec{a} + \vec{m_1})  G_{2\nu + \delta, \Delta} (\vec{a} + \vec{m_1})^t = (\vec{b} + \vec{m_2})  G_{2\nu + \delta, \Delta} (\vec{b} + \vec{m_2})^t = 0.
		\]

	\end{enumerate}
\end{lemma}
\begin{proof}
	The proof is analogous to the proof of Lemma \ref{Lifting Lemma} since $[\vec{a}]$ is adjacent to $[\vec{e}_1]$ if and only if $[\pi(\vec{a})]$ is adjacent to $[\pi(\vec{e}_1)]$.
\end{proof}

The results of $\mathcal{O}^{(2\nu + \delta)}_{2}(i), i = 1, 2$ are known and presented in Theorem 3.9 and Theorem 4.3 of \cite{GWZ16} when $\delta = 0$ or $2$. Since the graph $\mathcal{O}^{(2\nu + 1)}_{2}$ is isomorphic to the symplectic graph  $\Sp^{(2\nu + 1)}_{2}$ \cite{Zhe09}, its subconstituents have been studied in \cite{LW08}.  Applying these two results together with the above lemma lead us to the following two theorems for subconstituents $\mathcal{O}^{(2\nu + \delta)}_{2^n}(i), i = 1, 2$. Their proofs use similar argument to Theorem \ref{vertices}'s.

\begin{thm}
If $\delta = 0$ or $2$, then the subconstituent $\mathcal{O}^{(2\nu + \delta)}_{2^n}(1)$ is a quasi-strongly regular graph with parameters $(2^{n(2\nu + \delta - 2)}, k, \lambda, \{ c_1, c_2 \})$ where 
\begin{align*}
c_1 = k &= 2^{(n - 1)(2\nu + \delta - 2)}(2^{2\nu + \delta - 3} - (-1)^{\frac{\delta}{2}}2^{\frac{2\nu + \delta}{2} - 1} - (-1)^{\frac{\delta}{2} - 1}2^{\frac{2\nu + \delta}{2} - 2}), \\
 \lambda &= 2^{(n - 1)(2\nu + \delta - 2)}( 2^{2\nu + \delta - 4} - 2(-1)^{\frac{\delta}{2}}2^{\frac{2\nu + \delta}{2} - 1} + 3(-1)^{\frac{\delta}{2}}2^{\frac{2\nu + \delta}{2} - 2}), \\
 c_2 &= 2^{(n - 1)(2\nu + \delta - 2)}( 2^{2\nu + \delta - 4} - (-1)^{\frac{\delta}{2}}2^{\frac{2\nu + \delta}{2} - 1} - (-1)^{\frac{\delta}{2} - 1}2^{\frac{2\nu + \delta}{2} - 2}).
\end{align*}
If $\delta = 1$, then the subconstituent $\mathcal{O}^{(2 \nu + 1)}_{2^n}(1)$ is a $2^{n(2\nu - 2) + 1}$- regular graph with $2^{n (2\nu  - 1)}$ many vertices where any two adjacent vertices have $0$ or $2^{n(2\nu  - 2) -1 }$ common neighbors and any two non-adjacent vertices have $2^{n(2\nu  - 2) -1 }$ or $2^{n(2\nu - 2) + 1}$ common neighbors.
\end{thm} 

\begin{thm}
If $\delta = 0$ or $2$, then the subconstituent $\mathcal{O}^{(2 \cdot 2 + \delta)}_{2^n}(2)$ is a strongly regular with parameters $(v, k, \lambda, \mu)$ where
\begin{align*}
v &= 2^{(n - 1)(2 + \delta)}(2^{2 + \delta} + (-1)^{\frac{\delta}{2}}2^{\frac{\delta}{2} + 2} + (-1)^{\frac{\delta}{2} - 1}2^{\frac{\delta}{2} + 1} - 2), \\
\mu = k &= 2^{n(2 + \delta) - 1}, \\
\lambda &= 2^{(n - 1)(2 + \delta)}( 2^{1 + \delta} - 2^\delta - (-1)^{\frac{\delta}{2}}2^{\frac{\delta}{2} + 1} + (-1)^{\frac{\delta}{2}}2^{\frac{\delta}{2}}).
\end{align*}
Moreover, for $\nu \geq 3$, the subconstituent $\mathcal{O}^{(2 \nu + \delta)}_{2^n}(2)$ is a quasi-strongly regular with parameters $(v, k, \lambda, \{c_1, c_2 \})$ where 
\begin{align*}
v &= 2^{(n - 1)(2\nu + \delta - 2)}(2^{2\nu + \delta - 2} + (-1)^{\frac{\delta}{2}}2^{\frac{2\nu + \delta}{2}} + (-1)^{\frac{\delta}{2} - 1}2^{\frac{2\nu + \delta}{2} - 1} - 2), \\
c_1 = k &= 2^{n(2\nu + \delta - 2) - 1}, \\
\lambda &= 2^{(n - 1)(2\nu + \delta - 2)}(2^{2\nu + \delta - 4} - (-1)^{\frac{\delta}{2}}2^{\frac{2\nu + \delta}{2} - 1} + (-1)^{\frac{\delta}{2}}2^{\frac{2\nu + \delta}{2} - 2} ), \\
c_2 &= 2^{n(2\nu + \delta - 2) - 2}.
\end{align*}
If $\delta = 1$, then the subconstituent $\mathcal{O}^{(2 \nu + 1)}_{2^n}(2)$ is a strongly regular with parameters 
\[
(6 \cdot 2^{(n - 1)(2\nu - 1)}, 4 \cdot 2^{(n - 1)(2\nu - 1)}, 2 \cdot 2^{(n - 1)(2\nu - 1)}, 4 \cdot 2^{(n - 1)(2\nu - 1)}).
\]
Moreover, for $\nu \geq 3$, the subconstituent $\mathcal{O}^{(2 \nu + 1)}_{2^n}(1)$ is a $2^{n(2\nu - 1) - 1}$- regular graph with $2^{(n - 1)(2\nu - 1)} \\ (2^{2\nu - 1} - 2)$ many vertices where any two adjacent vertices have  $2^{n(2\nu  - 1) -2 }$ common neighbors and any two non-adjacent vertices have $2^{n(2\nu  - 1) -2 }$ or $2^{n(2\nu - 1) - 1}$ common neighbors.
\end{thm}

\end{document}